\documentclass[11pt]{article}

\usepackage{amsmath,amsthm,amscd,amsfonts,amssymb}
\usepackage{graphicx,slashbox}
\usepackage{subfigure}

\newtheorem{thm}{Theorem}[section]

\newtheorem{remark}[thm]{Remark}

\def\a{\alpha}

\def\LHI{{_a\mathcal{I}_t^\a}}
\def\RHI{{_t\mathcal{I}_b^\a}}
\def\LHD{{_a\mathcal{D}_t^\a}}
\def\RHD{{_t\mathcal{D}_b^\a}}
\def\LHIHz{{_1\mathcal{I}_t^{0.5}}}
\def\LHDHz{{_1\mathcal{D}_t^{0.5}}}

\newenvironment{keywords}{\begin{center}
\begin{minipage}[c]{13.4cm} {\bf Keywords:}} {\end{minipage}
\end{center}}

\newenvironment{msc}{\begin{center}
\begin{minipage}[c]{13.4cm} {\bf MSC 2010:}} {\end{minipage}
\end{center}}

\begin{document}

\title{Expansion formulas in terms of integer-order derivatives\\
for the Hadamard fractional integral and derivative\footnote{Part
of the first author's Ph.D., which is carried out at the University
of Aveiro under the \emph{Doctoral Program in Mathematics and Applications} (PDMA)
of Universities of Aveiro and Minho, supported by FCT fellowship SFRH/BD/33761/2009.\newline
Submitted 25-Jul-2011; revised 28-Nov-2011; accepted 02-Dec-2011; 
for publication in \emph{Numerical Functional Analysis and Optimization}.}}

\author{Shakoor Pooseh\\
\texttt{spooseh@ua.pt}
\and Ricardo Almeida\\
\texttt{ricardo.almeida@ua.pt}
\and Delfim F. M. Torres\\
\texttt{delfim@ua.pt}}

\date{Center for Research and Development in Mathematics and Applications\\
Department of Mathematics, University of Aveiro, 3810-193 Aveiro, Portugal}

\maketitle


\begin{abstract}
We obtain series expansion formulas for the Hadamard fractional integral
and fractional derivative of a smooth function. When considering finite sums only,
an upper bound for the error is given. Numerical simulations show
the efficiency of the approximation method.
\end{abstract}

\begin{msc}
26A33, 33F05.
\end{msc}



\section{Introduction}

In general terms, Fractional Calculus allows to define integrals and derivatives
of arbitrary real or complex order, and can be seen as a generalization of ordinary calculus.
A fractional derivative of order $\a>0$, when $\a$ is integer, coincides with
the classical derivative of order $n\in\mathbb N$, while
a fractional integral is an \textit{n-folded} integral. Although there exist
in the literature a large number of definitions for fractional operators
(integrals and derivatives), the Riemann--Liouville and Caputo are the most common
for fractional derivatives, and for fractional integrals the usual one
is the Riemann--Liouville definition. In our work we consider the Hadamard
fractional integral and fractional derivative. Although the definitions
go back to the works of Hadamard in 1892 \cite{Hadamard}, this type of operators
are not yet well studied and much exists to be done.

As well-known, for most problems involving fractional operators,
such as fractional differential equations or fractional control problems,
one cannot provide methods to compute the exact solutions analytically.
Therefore, numerical methods need to be employed.
Typically, problems are addressed by looking to the fractional operators
as special types of integrals and, using partitions of the domains, writing them
as finite sums, with some error for the final result.
Our approach is distinct from this, in the sense that we seek expansion formulas
for the Hadamard fractional operators with integer-order derivatives. In this way
we can rewrite the original problem, which depends on fractional operators, as a new one
that involves integer derivatives only and then successfully apply the standard methods
to obtain the desired solution. Also, in some cases, it can be easier
to find the fractional derivative or integral using those expansions, instead of applying
the direct definitions. In contrast with \cite{Atan2}, where an expansion for the
Riemann--Liouville fractional derivative is given, our expansions for the Hadamard
fractional derivative and integral do not omit the first derivative, allowing one
to obtain considerable better accuracy in computation.

The paper is organized in the following way. In Section~\ref{secpre} we review some concepts
of fractional calculus. Decomposition formulas for the left and right Hadamard fractional
integrals are given in Section~\ref{secDecomInt}, together with approximation formulas
and error estimations. Following the same approach, similar formulas are obtained
for the left and right Hadamard fractional derivatives in Section~\ref{secDecomDeriv}.
In Section~\ref{secExamples} we test the efficiency of such approximations with some examples,
comparing the analytical/exact solution with the numerical approximation.


\section{Preliminaries}
\label{secpre}

In this section we review some necessary definitions for our present work,
namely the Hadamard fractional integral and derivative. For more on fractional calculus,
we refer the interested reader to \cite{Kilbas,Miller,Podlubny,Samko}. For related work
on Hadamard fractional operators, see \cite{Butzer,Butzer2,Katugampola,Kilbas2,Kilbas3,Qian}.

Let $a,b$ be two reals with $0<a<b$ and $x:[a,b]\to\mathbb R$ be an integrable function.
The left and right Hadamard fractional integrals of order $\a>0$ are defined by
$$
\LHI x(t)=\frac{1}{\Gamma(\alpha)}\int_a^t
\left(\ln\frac{t}{\tau}\right)^{\alpha-1}\frac{x(\tau)}{\tau}d\tau,
\quad t\in ]a,b[,
$$
and
$$
\RHI x(t)=\frac{1}{\Gamma(\alpha)}
\int_t^b \left(\ln\frac{\tau}{t}\right)^{\a-1}\frac{x(\tau)}{\tau}d\tau,
\quad t\in ]a,b[,
$$
respectively. These integrals were introduced by Hadamard
in \cite{Hadamard} in the special case $a=0$.
When $\a=m$ is an integer, these fractional integrals
are \textit{m-folded} integrals
(see, \textrm{e.g.}, \cite{Butzer,Katugampola}):
$$
{_a\mathcal{I}_t^m} x(t)=\int_a^t \frac{d\tau_1}{\tau_1}\int_a^{\tau_1}
\frac{d\tau_2}{\tau_2}\cdots \int_a^{\tau_{m-1}} \frac{x(\tau_m)}{\tau_m}d\tau_m
$$
and
$$
{_t\mathcal{I}_b^m}x(t)=\int_t^b \frac{d\tau_1}{\tau_1}
\int_{\tau_1}^b \frac{d\tau_2}{\tau_2}
\cdots \int_{\tau_{m-1}}^b \frac{x(\tau_m)}{\tau_m}d\tau_m.
$$

For fractional derivatives, we also consider left and right operators.
For $\a>0$, the left and right Hadamard fractional derivatives
of order $\a$ are defined by
$$
\LHD x(t)=\left(t\frac{d}{dt}\right)^n\frac{1}{\Gamma(n-\alpha)}
\int_a^t \left(\ln\frac{t}{\tau}\right)^{n-\alpha-1}\frac{x(\tau)}{\tau}d\tau,
\quad t\in ]a,b[,
$$
and
$$
\RHD x(t)=\left(-t\frac{d}{dt}\right)^n\frac{1}{\Gamma(n-\alpha)}
\int_t^b \left(\ln\frac{\tau}{t}\right)^{n-\a-1}\frac{x(\tau)}{\tau}d\tau,
\quad t\in ]a,b[,
$$
respectively, with $n=[\a]+1$. When $\a=m$ is an integer,
we have (\textrm{cf.} \cite{Kilbas})
$$
{_a\mathcal{D}_t^m} x(t)=\left(t\frac{d}{dt}\right)^m x(t)
\mbox{ and } {_t\mathcal{D}_b^m}x(t)=\left(-t\frac{d}{dt}\right)^m x(t).
$$

Hadamard's fractional integrals and derivatives can be seen as inverse operations
of each other (see Property~2.28 and Theorem~2.3 of \cite{Kilbas}).

When $\a\in(0,1)$ and $x\in AC[a,b]$, with $AC[a,b]$ denoting
the space of all absolutely continuous functions $x:[a,b]\to\mathbb R$,
the Hadamard fractional derivatives may be expressed by
\begin{equation}
\label{LHDformula}
\LHD x(t)=\frac{x(a)}{\Gamma(1-\a)}\left(\ln\frac{t}{a}\right)^{-\alpha}
+\frac{1}{\Gamma(1-\alpha)}\int_a^t \left(\ln\frac{t}{\tau}\right)^{-\alpha}
\dot{x}(\tau)d\tau
\end{equation}
and
\begin{equation}
\label{RHDformula}
\RHD x(t)=\frac{x(b)}{\Gamma(1-\a)}\left(\ln\frac{b}{t}\right)^{-\alpha}
-\frac{1}{\Gamma(1-\alpha)}
\int_t^b \left(\ln\frac{\tau}{t}\right)^{-\a}\dot{x}(\tau)d\tau.
\end{equation}
For an arbitrary $\a>0$ we refer the reader to \cite[Theorem~3.2]{Kilbas2}.
If a function $x$ admits derivatives of any order,
then expansion formulas for the Hadamard fractional integrals and derivatives
of $x$, in terms of its integer-order derivatives,
are given in \cite[Theorem~17]{Butzer2}:
$$
{_0\mathcal{I}_t^\a}x(t)=\sum_{k=0}^\infty S(-\a,k)t^k x^{(k)}(t)
$$
and
$$
{_0\mathcal{D}_t^\a}x(t)=\sum_{k=0}^\infty S(\a,k)t^k x^{(k)}(t),
$$
where
$$
S(\a,k)=\frac{1}{k!}\sum_{j=1}^k(-1)^{k-j} {k \choose j} j^{\a}
$$
is the Stirling function.


\section{An expansion formula for the Hadamard fractional integral}
\label{secDecomInt}

In this section we consider the class of differentiable functions
up to order $n+1$, $x\in C^{n+1}[a,b]$, and deduce expansion formulas
for the Hadamard fractional integrals in terms of $x^{(i)}(\cdot)$,
for $i\in \{0,\ldots,n\}$. Before presenting the result in its full extension,
we briefly explain the techniques involved for the particular case $n=2$.
To that purpose, let $x\in C^3[a,b]$.
Integrating by parts three times, we obtain
$$
\begin{array}{ll}
\LHI x(t)&=\displaystyle-\frac{1}{\Gamma(\a)}
\int_a^t-\frac{1}{\tau} \left(\ln\frac{t}{\tau}\right)^{\a-1}x(\tau)d\tau\\
&=\displaystyle\frac{1}{\Gamma(\a+1)} \left(\ln\frac{t}{a}\right)^{\a}x(a)
-\frac{1}{\Gamma(\a+1)}\int_a^t-\frac{1}{\tau}
\left(\ln\frac{t}{\tau}\right)^{\a}\tau \dot{x}(\tau)d\tau\\
&=\displaystyle\frac{1}{\Gamma(\a+1)} \left(\ln\frac{t}{a}\right)^{\a}x(a)
+\frac{1}{\Gamma(\a+2)} \left(\ln\frac{t}{a}\right)^{\a+1}a\dot{x}(a)\\
&\quad\displaystyle -\frac{1}{\Gamma(\a+2)}\int_a^t-\frac{1}{\tau}
\left(\ln\frac{t}{\tau}\right)^{\a+1}(\tau \dot{x}(\tau)+\tau^2\ddot{x}(\tau))d\tau\\
&=\displaystyle\frac{1}{\Gamma(\a+1)} \left(\ln\frac{t}{a}\right)^{\a}x(a)
+\frac{1}{\Gamma(\a+2)} \left(\ln\frac{t}{a}\right)^{\a+1}a\dot{x}(a)\\
&\displaystyle\quad+\frac{1}{\Gamma(\a+3)}
\left(\ln\frac{t}{a}\right)^{\a+2}(a\dot{x}(a) + a^2\ddot{x}(a))\\
&\quad+\displaystyle\frac{1}{\Gamma(\a+3)}
\int_a^t \left(\ln\frac{t}{\tau}\right)^{\a+2}(\dot{x}(\tau)
+3\tau \ddot{x}(\tau)+\tau^2\dddot{x}(\tau))d\tau.
\end{array}
$$
On the other hand, using the binomial formula, we have
$$
\begin{array}{ll}
\displaystyle\left(\ln\frac{t}{\tau}\right)^{\a+2}
&=\displaystyle\left(\ln\frac{t}{a}\right)^{\a+2}\left(
1-\frac{\ln\frac{\tau}{a}}{\ln\frac{t}{a}}\right)^{\a+2}\\
&=\displaystyle\left(\ln\frac{t}{a}\right)^{\a+2}
\sum_{p=0}^\infty\frac{\Gamma(p-\a-2)}{\Gamma(-\a-2)p!}
\cdot \frac{\left(\ln\frac{\tau}{a}\right)^p}{\left(\ln\frac{t}{a}\right)^p}.
\end{array}
$$
This series converges since $\tau\in[a,t]$ and $\a+2>0$.
Combining these formulas, we get
\begin{multline*}
\LHI x(t)=\frac{1}{\Gamma(\a+1)} \left(\ln\frac{t}{a}\right)^{\a}x(a)
+\frac{1}{\Gamma(\a+2)} \left(\ln\frac{t}{a}\right)^{\a+1}a\dot{x}(a)
+\frac{1}{\Gamma(\a+3)} \left(\ln\frac{t}{a}\right)^{\a+2}(a\dot{x}(a)
+a^2\ddot{x}(a))\\
+\frac{1}{\Gamma(\a+3)}\left(\ln\frac{t}{a}\right)^{\a+2}
\sum_{p=0}^\infty\frac{\Gamma(p-\a-2)}{\Gamma(-\a-2)p!\left(\ln\frac{t}{a}\right)^p}
\int_a^t \left(\ln\frac{\tau}{a}\right)^p \left(\dot{x}(\tau)+3\tau \ddot{x}(\tau)
+\tau^2\dddot{x}(\tau)\right)d\tau.
\end{multline*}
Now, split the series into the two cases $p=0$ and $p=1\ldots\infty$,
and integrate by parts the second one. We obtain
$$
\begin{array}{ll}
\LHI x(t)&=\displaystyle\frac{1}{\Gamma(\a+1)}
\left(\ln\frac{t}{a}\right)^{\a}x(a)
+\frac{1}{\Gamma(\a+2)} \left(\ln\frac{t}{a}\right)^{\a+1}a\dot{x}(a)\\
&\displaystyle\quad +\frac{1}{\Gamma(\a+3)}
\left(\ln\frac{t}{a}\right)^{\a+2}(t\dot{x}(t)
+t^2\ddot{x}(t))\left[1+\sum_{p=1}^\infty\frac{\Gamma(p-\a-2)}{\Gamma(-\a-2)p!} \right]\\
&\displaystyle\quad +\frac{1}{\Gamma(\a+2)}\left(\ln\frac{t}{a}\right)^{\a+2}
\sum_{p=1}^\infty\frac{\Gamma(p-\a-2)}{\Gamma(-\a-1)(p-1)!\left(\ln\frac{t}{a}\right)^p}
\int_a^t \left(\ln\frac{\tau}{a}\right)^{p-1} (\dot{x}(\tau)+\tau \ddot{x}(\tau))d\tau.
\end{array}
$$
Repeating this procedure two more times, we obtain the following:
$$
\begin{array}{ll}
\LHI x(t)&=\displaystyle \frac{1}{\Gamma(\a+1)}
\left(\ln\frac{t}{a}\right)^{\a}x(t)\left[1
+\sum_{p=3}^\infty\frac{\Gamma(p-\a-2)}{\Gamma(-\a)(p-2)!} \right]\\
&\displaystyle\quad +\frac{1}{\Gamma(\a+2)}
\left(\ln\frac{t}{a}\right)^{\a+1}t\dot{x}(t)\left[1
+\sum_{p=2}^\infty\frac{\Gamma(p-\a-2)}{\Gamma(-\a-1)(p-1)!} \right]\\
&\displaystyle\quad +\frac{1}{\Gamma(\a+3)} \left(\ln\frac{t}{a}\right)^{\a+2}(t\dot{x}(t)
+t^2\ddot{x}(t))\left[1+\sum_{p=1}^\infty\frac{\Gamma(p-\a-2)}{\Gamma(-\a-2)p!} \right]\\
&\displaystyle\quad +\frac{1}{\Gamma(\a)}  \left(\ln\frac{t}{a}\right)^{\a+2}
\sum_{p=3}^\infty\frac{\Gamma(p-\a-2)}{\Gamma(-\a+1)(p-3)!\left(\ln\frac{t}{a}\right)^p}
\int_a^t \left(\ln\frac{\tau}{a}\right)^{p-3} \frac{x(\tau)}{\tau}d\tau,
\end{array}
$$
or, in a more concise way,
$$
\begin{array}{ll}
\LHI x(t)&=\displaystyle A_0(\alpha)\left(\ln\frac{t}{a}\right)^{\a} x(t)
+ A_1(\alpha)\left(\ln\frac{t}{a}\right)^{\a+1}t\dot{x}(t)\\
&\quad\displaystyle + A_2(\alpha) \left(\ln\frac{t}{a}\right)^{\a+2}(t\dot{x}(t)
+t^2\ddot{x}(t))+ \sum_{p=3}^\infty B(\a,p)\left(\ln\frac{t}{a}\right)^{\a+2-p}V_p(t)
\end{array}
$$
with
$$
\begin{array}{ll}
A_0(\a)&=\displaystyle\frac{1}{\Gamma(\a+1)}\left[1
+\sum_{p=3}^\infty\frac{\Gamma(p-\a-2)}{\Gamma(-\a)(p-2)!}\right],\\
A_1(\a)&=\displaystyle\frac{1}{\Gamma(\a+2)}\left[1
+\sum_{p=2}^\infty\frac{\Gamma(p-\a-2)}{\Gamma(-\a-1)(p-1)!}\right],\\
A_2(\a)&=\displaystyle\frac{1}{\Gamma(\a+3)}\left[1
+\sum_{p=1}^\infty\frac{\Gamma(p-\a-2)}{\Gamma(-\a-2)p!}\right],
\end{array}
$$
\begin{equation}
\label{Case:B3}
B(\a,p)=\displaystyle\frac{\Gamma(p-\a-2)}{\Gamma(\a)\Gamma(1-\a)(p-2)!}
\end{equation}
and
\begin{equation}
\label{Case:Vp3}
V_p(t)=\int_a^t (p-2)\left(\ln\frac{\tau}{a}\right)^{p-3}\frac{x(\tau)}{\tau}d\tau,
\end{equation}
where we assume the series and the integral $V_p$ to be convergent.

\begin{remark}
When useful, namely on fractional differential equations problems,
we can define $V_p$ as in \eqref{Case:Vp3}
by the the solution of the system
$$
\left\{
\begin{array}{l}
\dot{V_p}(t)=(p-2)\left(
\ln\frac{t}{a}\right)^{p-3}\frac{x(t)}{t}\\
V_p(a)=0
\end{array}
\right.
$$
for all $p=3,4,\ldots$
\end{remark}

We now discuss the convergence of the series involved
in the definitions of $A_i(\a)$, for $i \in \{0,1,2\}$.
Simply observe that
$$
\sum_{p=3-i}^\infty\frac{\Gamma(p-\a-2)}{\Gamma(-\a-i)(p-2+i)!}
={_1F_0} (-\a-i,1)-1,
$$
and ${_1F_0}(a,x)$ converges absolutely when $|x|=1$
if $a<0$ (\cite[Theorem~2.1.2]{Andrews}).

For numerical purposes, only finite sums are considered,
and thus the Hadamard left fractional integral is approximated by the decomposition
\begin{equation}
\begin{array}{ll}
\label{Case:n=3}
\LHI x(t)&\approx\displaystyle A_0(\alpha,N)\left(\ln\frac{t}{a}\right)^{\a} x(t)
+ A_1(\alpha,N)\left(\ln\frac{t}{a}\right)^{\a+1}t\dot{x}(t)\\
&\quad\displaystyle + A_2(\alpha,N) \left(\ln\frac{t}{a}\right)^{\a
+2}(t\dot{x}(t)+t^2\ddot{x}(t))+ \sum_{p=3}^N B(\a,p)\left(
\ln\frac{t}{a}\right)^{\a+2-p}V_p(t)
\end{array}
\end{equation}
with
$$
\begin{array}{ll}
A_0(\a,N)&=\displaystyle\frac{1}{\Gamma(\a+1)}\left[1
+\sum_{p=3}^N\frac{\Gamma(p-\a-2)}{\Gamma(-\a)(p-2)!}\right],\\
A_1(\a,N)&=\displaystyle\frac{1}{\Gamma(\a+2)}\left[1
+\sum_{p=2}^N\frac{\Gamma(p-\a-2)}{\Gamma(-\a-1)(p-1)!}\right],\\
A_2(\a,N)&=\displaystyle\frac{1}{\Gamma(\a+3)}\left[1
+\sum_{p=1}^N\frac{\Gamma(p-\a-2)}{\Gamma(-\a-2)p!}\right],
\end{array}
$$
$B(\a,p)$ and $V_p(t)$ as in \eqref{Case:B3}--\eqref{Case:Vp3}, and $N\geq3$.
We proceed with an estimation for the error on such approximation.
We have proven before that
$$
\LHI x(t)=\frac{1}{\Gamma(\a+1)} \left(\ln\frac{t}{a}\right)^{\a}x(a)
+\frac{1}{\Gamma(\a+2)} \left(\ln\frac{t}{a}\right)^{\a+1}a\dot{x}(a)
+\frac{1}{\Gamma(\a+3)} \left(\ln\frac{t}{a}\right)^{\a+2}(a\dot{x}(a)
+a^2\ddot{x}(a))
$$
$$
+\frac{1}{\Gamma(\a+3)}  \left(\ln\frac{t}{a}\right)^{\a+2} \int_a^t
\sum_{p=0}^\infty\frac{\Gamma(p-\a-2)}{\Gamma(-\a-2)p!}\frac{\left(
\ln\frac{\tau}{a}\right)^p}{\left(\ln\frac{t}{a}\right)^p}(\dot{x}(\tau)
+3\tau \ddot{x}(\tau)+\tau^2\dddot{x}(\tau))d\tau.
$$
When we consider finite sums up to order $N$, the error is given by
$$
\left| E_{tr}(t)\right|=\left|\frac{1}{\Gamma(\a+3)}
\left(\ln\frac{t}{a}\right)^{\a+2}
\int_a^t R_N(\tau) (\dot{x}(\tau)+3\tau \ddot{x}(\tau)
+\tau^2\dddot{x}(\tau))d\tau\right|
$$
with
$$
R_N(\tau)=\sum_{p=N+1}^\infty\frac{\Gamma(p-\a-2)}{\Gamma(-\a-2)p!}\frac{
\left(\ln\frac{\tau}{a}\right)^p}{\left(\ln\frac{t}{a}\right)^p}.
$$
Since $\tau\in[a,t]$, we have
$$
\begin{array}{ll}
|R_N(\tau)| &  \displaystyle \leq \sum_{p=N+1}^\infty\left|
\binom{\a+2}{p} \right|
\leq  \sum_{p=N+1}^\infty \frac{e^{(\a+2)^2+\a+2}}{p^{\a+3}}\\
&\displaystyle\leq \int_{N}^\infty \frac{e^{(\a+2)^2+\a+2}}{p^{\a+3}}dp
=\frac{e^{(\a+2)^2+\a+2}}{(\a+2)N^{\a+2}}.
\end{array}
$$
Therefore,
$$
\left| E_{tr}(t)\right|\leq \frac{1}{\Gamma(\a+3)}
\left(\ln\frac{t}{a}\right)^{\a+2} \frac{e^{(\a+2)^2
+\a+2}}{(\a+2)N^{\a+2}}\left[(t-a)L_1(t)+3(t-a)^2L_2(t)+(t-a)^3L_3(t)\right],
$$
where
$$
L_i(t)=\max_{\tau\in[a,t]}|x^{(i)}(\tau)|,
\quad i \in \{1,2,3\}.
$$

Following similar arguments as done for $n=2$,
we can prove the general case with an expansion
up to the derivative of order $n$. First,
we introduce a notation. Given $k\in\mathbb N \cup\{0\}$,
we define the sequences $x_{k,0}(t)$ and $x_{k,1}(t)$
recursively by the formulas
$$
x_{0,0}(t)=x(t) \mbox{ and } x_{k+1,0}(t)
=t\frac{d}{dt} x_{k,0}(t),
\mbox{ for } k\in\mathbb N\cup\{0\},
$$
and
$$
x_{0,1}(t)=\dot{x}(t) \mbox{ and } x_{k+1,1}(t)
=\frac{d}{dt} (t x_{k,1}(t)),
\mbox{ for } k\in\mathbb N\cup\{0\}.
$$

\begin{thm}
Let $n\in\mathbb N$, $0<a<b$ and $x:[a,b]\to\mathbb R$
be a function of class $C^{n+1}$. Then,
$$
\LHI x(t)=\sum_{i=0}^{n}A_i(\alpha)\left(\ln\frac{t}{a}\right)^{\a+i} x_{i,0}(t)
+\sum_{p=n+1}^\infty B(\a,p)\left(\ln\frac{t}{a}\right)^{\a+n-p}V_p(t)
$$
with
$$
\begin{array}{ll}
A_i(\a)&=\displaystyle\frac{1}{\Gamma(\a+i+1)}\left[1
+\sum_{p=n-i+1}^\infty\frac{\Gamma(p-\a-n)}{\Gamma(-\a-i)(p-n+i)!}\right],\\
B(\a,p)&=\displaystyle\frac{\Gamma(p-\a-n)}{\Gamma(\a)\Gamma(1-\a)(p-n)!},\\
V_p(t)&=\int_a^t (p-n)\left(\ln\frac{\tau}{a}\right)^{p-n-1}\frac{x(\tau)}{\tau}d\tau.
\end{array}
$$
Moreover, if we consider the approximation
$$
\LHI x(t)\approx\sum_{i=0}^{n}A_i(\alpha,N)\left(\ln\frac{t}{a}\right)^{\a+i}
x_{i,0}(t)+\sum_{p=n+1}^N B(\a,p)\left(\ln\frac{t}{a}\right)^{\a+n-p}V_p(t)
$$
with $N \geq n+1$ and
$$
A_i(\a,N)=\frac{1}{\Gamma(\a+i+1)}\left[1
+\sum_{p=n-i+1}^N\frac{\Gamma(p-\a-n)}{\Gamma(-\a-i)(p-n+i)!}\right],
$$
then the error is bounded by the expression
$$
\left| E_{tr}(t)\right|\leq L_n(t)\frac{e^{(\a+n)^2+\a+n}}{
\Gamma(\a+n+1)(\a+n)N^{\a+n}} \left(\ln\frac{t}{a}\right)^{\a+n}(t-a),
$$
where
$$
L_n(t)=\max_{\tau\in[a,t]}|x_{n,1}(\tau)|.
$$
\end{thm}

\begin{proof}
Applying integration by parts repeatedly
and the binomial formula, we arrive to
$$
\begin{array}{ll}
\LHI x(t)&=\displaystyle\sum_{i=0}^{n}\frac{1}{\Gamma(\a+i+1)}
\left(\ln\frac{t}{a}\right)^{\a+i}x_{i,0}(a)\\
&\quad\displaystyle+\frac{1}{\Gamma(\a+n+1)}
\left(\ln\frac{t}{a}\right)^{\a+n}
\sum_{p=0}^\infty\frac{\Gamma(p-\a-n)}{\Gamma(-\a-n)p!\left(
\ln\frac{t}{a}\right)^p} \int_a^t \left(\ln\frac{\tau}{a}\right)^p x_{n,1}(\tau)d\tau.
\end{array}
$$
To achieve the expansion formula, we repeat the same procedure as for the case $n=2$:
we split the sum into two parts (the first term plus the remainings)
and integrate by parts the second one. The convergence of the series
$A_i(\a)$ is ensured by the relation
$$
\sum_{p=n-i+1}^\infty\frac{\Gamma(p-\a-n)}{\Gamma(-\a-i)(p-n+i)!}
={_1F_0} (-\a-i,1)-1.
$$
The error on the approximation is given by
$$
\left| E_{tr}(t)\right|=\left|\frac{1}{\Gamma(\a+n+1)}\left(
\ln\frac{t}{a}\right)^{\a+n} \int_a^t R_N(\tau)x_{n,1}(\tau)d\tau\right|
$$
with
$$
R_N(\tau)=\sum_{p=N+1}^\infty\frac{\Gamma(p-\a-n)}{\Gamma(-\a-n)p!}
\frac{\left(\ln\frac{\tau}{a}\right)^p}{\left(\ln\frac{t}{a}\right)^p}.
$$
Also, for $\tau\in[a,t]$,
$$
|R_N(\tau)| \leq \sum_{p=N+1}^\infty\left| \binom{\a+n}{p} \right|
\leq \frac{e^{(\a+n)^2+\a+n}}{(\a+n)N^{\a+n}}.
$$
\end{proof}

We remark that the error formula tends to zero as $N$ increases.
Similarly to what was done with the left fractional integral,
we can also expand the right Hadamard fractional integral.

\begin{thm}
\label{HRFI}
Let $n\in\mathbb N$, $0<a<b$ and $x:[a,b]\to\mathbb R$
be a function of class $C^{n+1}$. Then,
$$
\RHI x(t)=\sum_{i=0}^{n}A_i(\alpha)\left(\ln\frac{b}{t}\right)^{\a+i} x_{i,0}(t)
+\sum_{p=n+1}^\infty B(\a,p)\left(\ln\frac{b}{t}\right)^{\a+n-p}W_p(t)
$$
with
$$
\begin{array}{ll}
A_i(\a)&=\displaystyle\frac{(-1)^i}{\Gamma(\a+i+1)}\left[1
+\sum_{p=n-i+1}^\infty\frac{\Gamma(p-\a-n)}{\Gamma(-\a-i)(p-n+i)!}\right],\\
B(\a,p)&=\displaystyle\frac{\Gamma(p-\a-n)}{\Gamma(\a)\Gamma(1-\a)(p-n)!},\\
W_p(t)&=\displaystyle\int_t^b (p-n)\left(
\ln\frac{b}{\tau}\right)^{p-n-1}\frac{x(\tau)}{\tau}d\tau.
\end{array}
$$
\end{thm}

\begin{remark}
Analogously to what was done for the left fractional integral,
one can consider an approximation for the right fractional integral
by considering finite sums in the expansion obtained in Theorem~\ref{HRFI}.
In this case, the error is bounded by
$$
\left| E_{tr}(t)\right|\leq L_n(t)\frac{e^{(\a+n)^2
+\a+n}}{\Gamma(\a+n+1)(\a+n)N^{\a+n}}
\left(\ln\frac{b}{t}\right)^{\a+n}(b-t),
$$
where
$$
L_n(t)=\max_{\tau\in[t,b]}|x_{n,1}(\tau)|.
$$
\end{remark}


\section{An expansion formula for the Hadamard fractional derivative}
\label{secDecomDeriv}

Starting with formulas \eqref{LHDformula} and \eqref{RHDformula},
and applying similar techniques as presented in Section~\ref{secDecomInt},
we are able to present expansion formulas, and respective approximation
formulas with an error estimation, for the left and right
Hadamard fractional derivatives. Due to restrictions on the number of pages,
we omit the details here and just exhibit the results.

Given $n\in\mathbb N$ and $x\in C^{n+1}[a,b]$, we have
$$
\begin{array}{ll}
\LHD x(t)& =\displaystyle \frac{1}{\Gamma(1-\a)}\left(\ln\frac{t}{a}\right)^{-\a}x(t)
+\sum_{i=1}^{n}A_i(\alpha)\left(\ln\frac{t}{a}\right)^{i-\a} x_{i,0}(t)\\
&\quad\displaystyle +\sum_{p=n+1}^\infty \left[
B(\a,p)\left(\ln\frac{t}{a}\right)^{n-\a-p}V_p(t)
+\frac{\Gamma(p+\a-n)}{\Gamma(1-\a)\Gamma(\a)(p-n)!}\left(
\ln\frac{t}{a}\right)^{-\a}x(t)\right]
\end{array}
$$
with
$$
\begin{array}{ll}
A_i(\a)&=\displaystyle\frac{1}{\Gamma(i+1-\a)}\left[1
+\sum_{p=n-i+1}^\infty\frac{\Gamma(p+\a-n)}{\Gamma(\a-i)(p-n+i)!}\right],
\quad i \in \{1,\ldots,n\}, \\
B(\a,p)&=\displaystyle\frac{\Gamma(p+\a-n)}{\Gamma(-\a)\Gamma(1+\a)(p-n)!},
\quad p \in \{n+1,\ldots\}, \\
V_p(t)&=\int_a^t (p-n)\left(\ln\frac{\tau}{a}\right)^{p-n-1}\frac{x(\tau)}{\tau}d\tau,
\quad p \in \{n+1,\ldots\}.
\end{array}
$$
When we consider finite sums,
$$
\LHD x(t)\approx \sum_{i=0}^{n}A_i(\alpha,N)\left(\ln\frac{t}{a}\right)^{i-\a}
x_{i,0}(t)+\sum_{p=n+1}^N B(\a,p)\left(\ln\frac{t}{a}\right)^{n-\a-p}V_p(t)
$$
with
$$
A_i(\a,N)=\displaystyle\frac{1}{\Gamma(i+1-\a)}\left[1
+\sum_{p=n-i+1}^N\frac{\Gamma(p+\a-n)}{\Gamma(\a-i)(p-n+i)!}\right],
\quad i \in\{0,\ldots,n\},
$$
and the error is bounded by
$$
\left| E_{tr}(t)\right|\leq L_n(t)\frac{e^{(n-\a)^2
+n-\a}}{\Gamma(n+1-\a)(n-\a)N^{n-\a}}
\left(\ln\frac{t}{a}\right)^{n-\a}(t-a),
$$
where
$$
L_n(t)=\max_{\tau\in[a,t]}|x_{n,1}(\tau)|.
$$

\begin{remark}
The series involved in the definition of $A_i$ are convergent,
for all $i \in \{1,\ldots,n\}$. This is due to the fact that
$$
\sum_{p=n-i+1}^\infty\frac{\Gamma(p+\a-n)}{\Gamma(\a-i)(p-n+i)!}
={_1F_0}(\a-i,1)-1
$$
and ${_1F_0}(\a-i,1)$ converges
(\cite[Theorem~2.1.1]{Andrews}), since $\a-i<0$.
\end{remark}

\begin{remark}
For the right Hadamard fractional derivative, the expansion reads as
$$
\begin{array}{ll}
\RHD x(t)& =\displaystyle \frac{1}{\Gamma(1-\a)}\left(\ln\frac{b}{t}\right)^{-\a}x(t)
+\sum_{i=1}^{n}A_i(\alpha)\left(\ln\frac{b}{t}\right)^{i-\a} x_{i,0}(t)\\
&\quad\displaystyle +\sum_{p=n+1}^\infty\left[
B(\a,p)\left(\ln\frac{b}{t}\right)^{n-\a-p}W_p(t)
+\frac{\Gamma(p+\a-n)}{\Gamma(1-\a)\Gamma(\a)(p-n)!}\left(
\ln\frac{b}{t}\right)^{-\a}x(t)\right]
\end{array}
$$
with
$$
\begin{array}{ll}
A_i(\a)&=\displaystyle\frac{(-1)^i}{\Gamma(i+1-\a)}\left[1
+\sum_{p=n-i+1}^\infty\frac{\Gamma(p+\a-n)}{\Gamma(\a-i)(p-n+i)!}\right],
\quad i \in \{1,\ldots,n\}, \\
B(\a,p)&=\displaystyle\frac{\Gamma(p+\a-n)}{\Gamma(-\a)\Gamma(1+\a)(p-n)!},
\quad p \in \{n+1,\ldots\},\\
W_p(t)&=\displaystyle\int_t^b (p-n)\left(
\ln\frac{b}{\tau}\right)^{p-n-1}\frac{x(\tau)}{\tau}d\tau.
\end{array}
$$
\end{remark}


\section{Examples}
\label{secExamples}

We obtained approximation formulas for the Hadamard fractional integrals
and derivatives, and an upper bound for the error on such decompositions.
In this section we study several cases, comparing the solution with the approximations.
To gather more information on the accuracy, we evaluate the error using the distance
$$
\mbox{dist}=\sqrt{\int_a^b (X(t)-\tilde{X}(t))^2dt},
$$
where $X(t)$ is the exact formula and $\tilde{X}(t)$ the approximation.
To begin with, we consider $\a=0.5$ and functions $x_1(t)=\ln t$
and $x_2(t)=1$ with $t\in[1,10]$. Then,
$$
\LHIHz x_1(t)=\frac{\sqrt{\ln^3 t}}{\Gamma(2.5)}
\mbox{ and } \LHIHz x_2(t)
=\frac{\sqrt{\ln t}}{\Gamma(1.5)}
$$
(\textrm{cf.} \cite[Property~2.24]{Kilbas}).
We consider the expansion formula for $n=2$
as in \eqref{Case:n=3} for both cases.
We obtain then the approximations
$$
\LHIHz x_1(t)\approx \left[ A_0(0.5,N)+A_1(0.5,N)
+\sum_{p=3}^N B(0.5,p)\frac{p-2}{p-1}\right]\sqrt{\ln^3 t}
$$
and
$$
\LHIHz x_2(t)\approx \left[ A_0(0.5,N)
+\sum_{p=3}^N B(0.5,p)\right]\sqrt{\ln t}.
$$
The results are exemplified in Figures~\ref{ExpLnt}
and \ref{Exp1}. As can be seen, the value $N=3$ is enough
in order to obtain a good accuracy in the sense of the error function.

\begin{figure}[ht!]
\begin{center}
\subfigure[$\LHIHz(\ln t)$]{\label{ExpLnt}\includegraphics[scale=0.5]{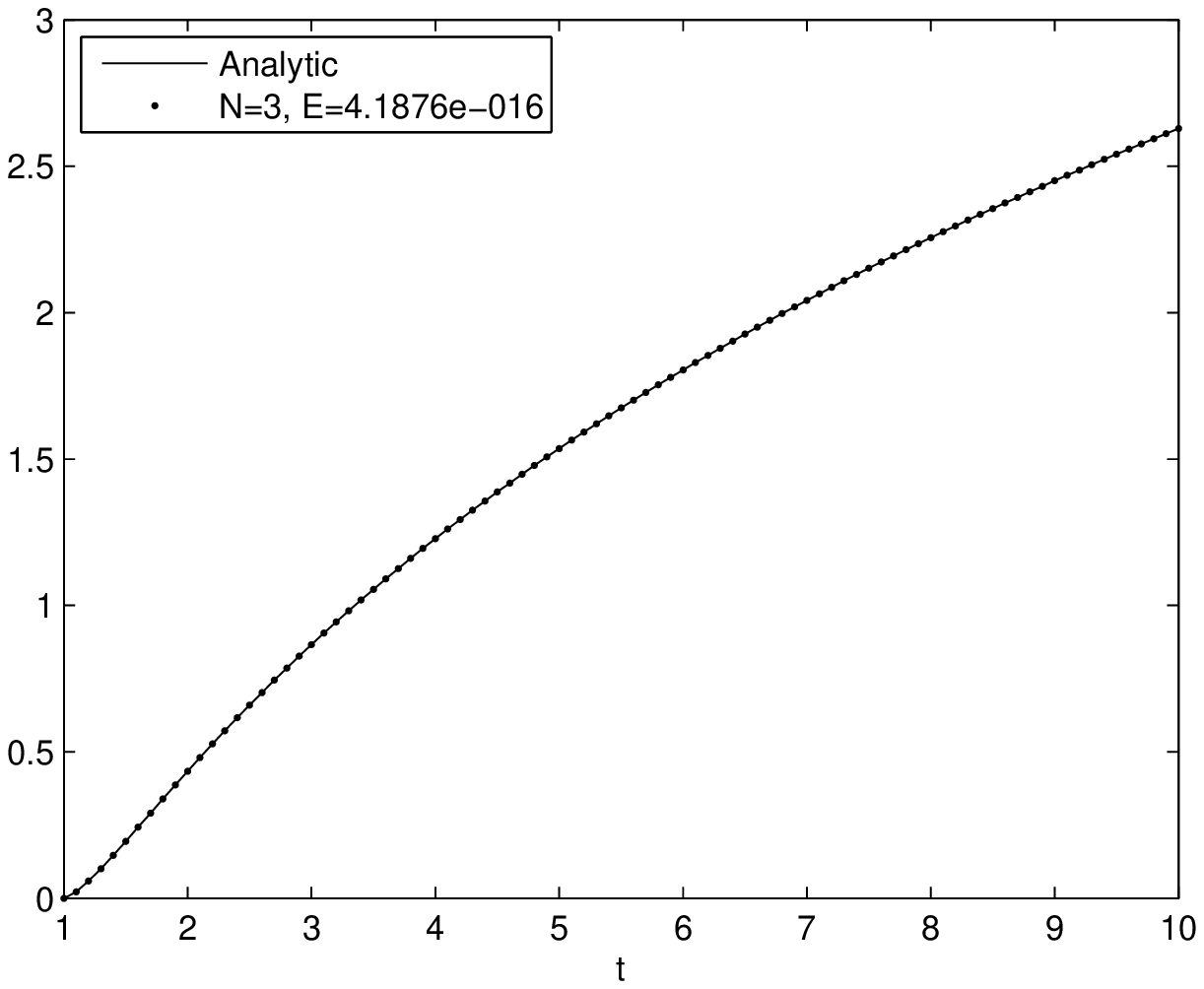}}
\subfigure[$\LHIHz(1)$]{\label{Exp1}\includegraphics[scale=0.5]{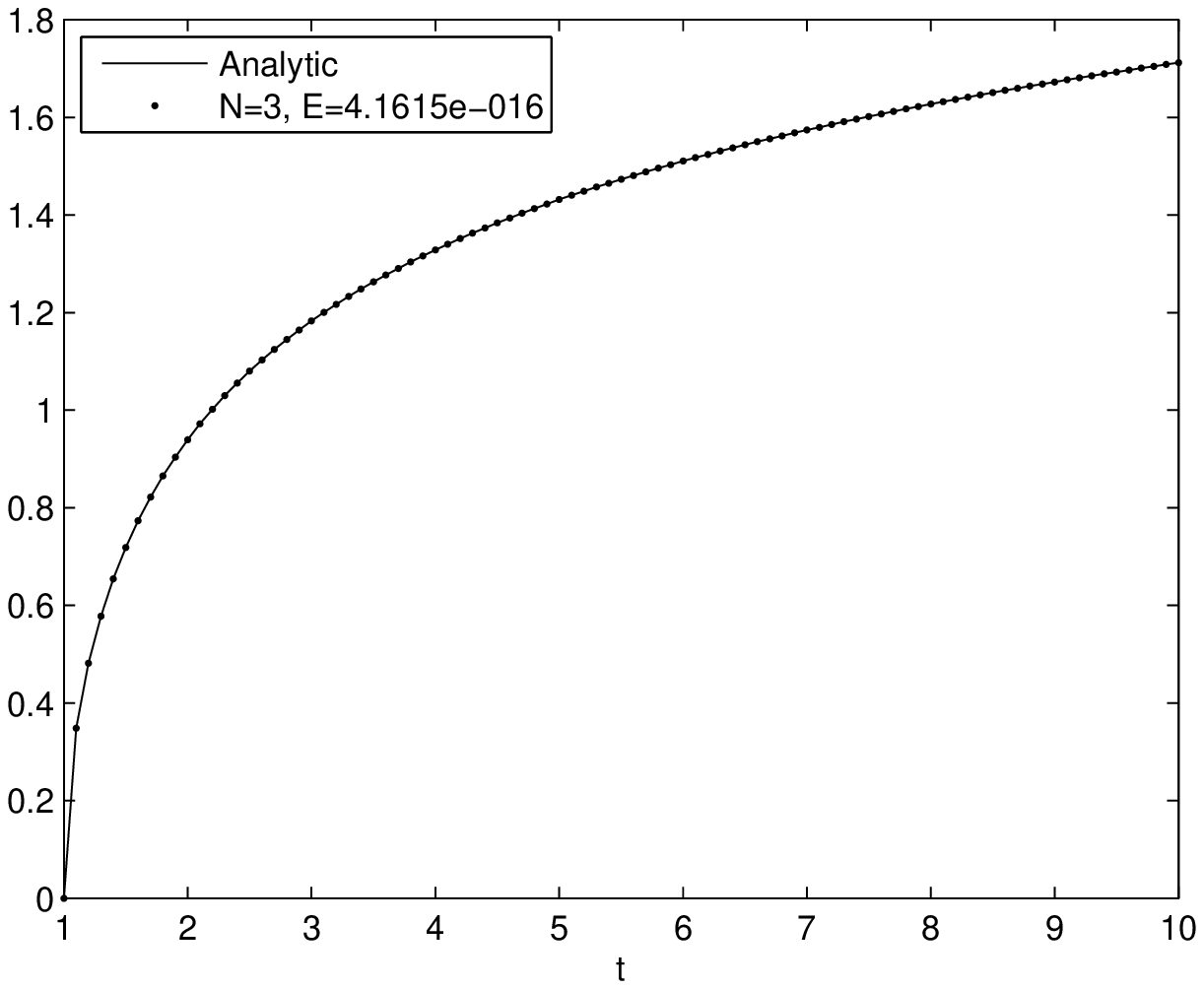}}
\end{center}
\caption{Analytic vs. numerical approximation for $n=2$.}
\end{figure}

We now test the approximation on the power functions $x_3(t)=t^4$
and $x_4(t)=t^9$, with $t\in[1,2]$. Observe first that
$$
\LHIHz(t^k)=\frac{1}{\Gamma(0.5)}\int_1^t
\left(\ln\frac{t}{\tau}\right)^{-0.5}\tau^{k-1}d\tau
=\frac{t^k}{\Gamma(0.5)}\int_0^{\ln t} \xi^{-0.5}e^{-\xi k} d\xi
$$
by the change of variables $\xi=\ln\frac{t}{\tau}$. In our cases,
$$
\LHIHz(t^4)\approx \frac{0.8862269255}{\Gamma(0.5)}t^4 \mbox{erf}(2\sqrt{\ln t})
\mbox { and } \LHIHz(t^9)\approx \frac{0.5908179503}{\Gamma(0.5)}t^9
\mbox{erf}(3\sqrt{\ln t}),
$$
where $\mbox{erf}(\cdot)$ is the error function. In Figures~\ref{Expt4}
and \ref{Expt9} we show approximations for several values of $N$.
We mention that, as $N$ increases, the error decreases
and thus we obtain a better approximation.

\begin{figure}[ht!]
\begin{center}
\subfigure[$\LHIHz(t^4)$]{\label{Expt4}\includegraphics[scale=0.5]{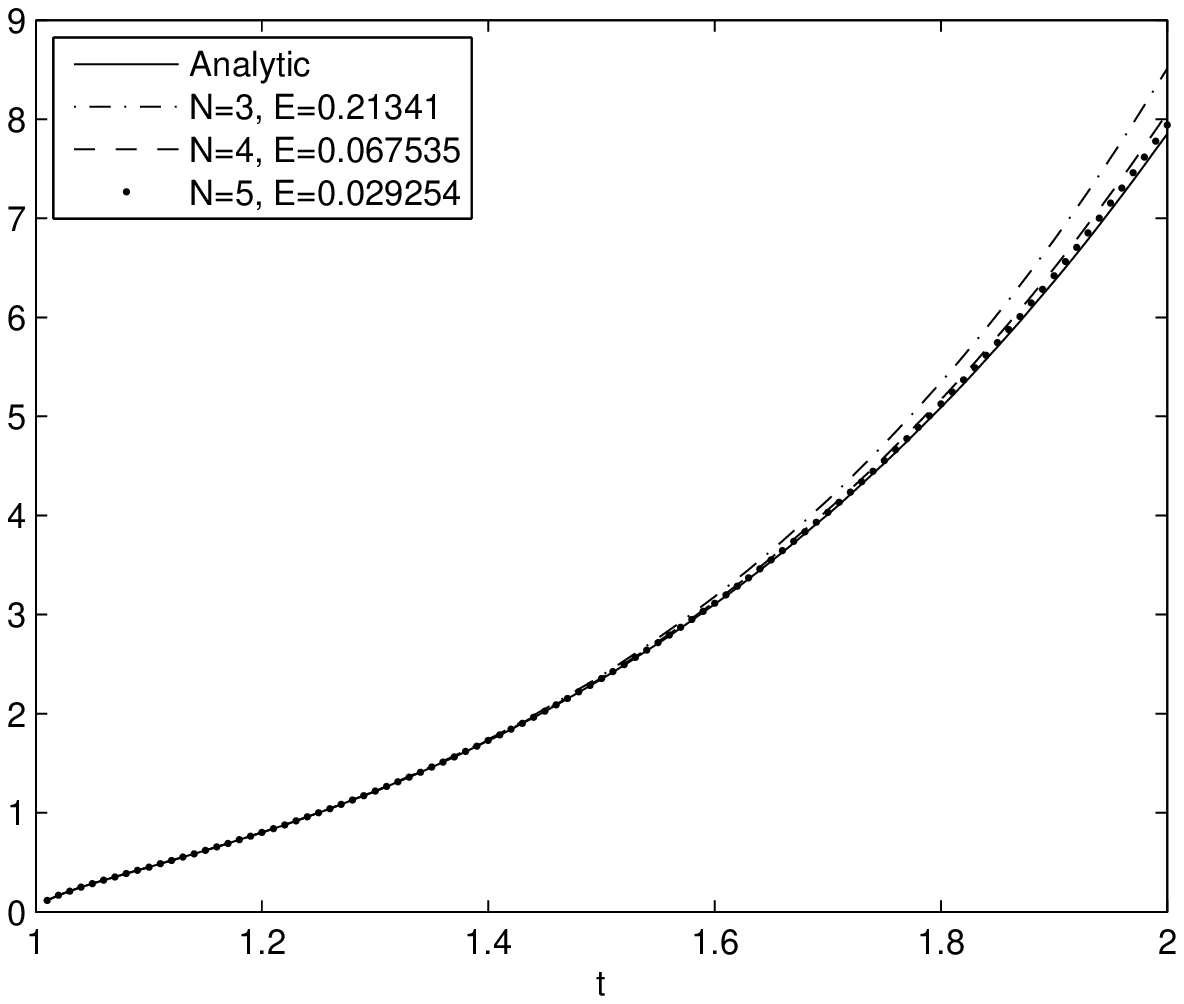}}
\subfigure[$\LHIHz(t^9)$]{\label{Expt9}\includegraphics[scale=0.5]{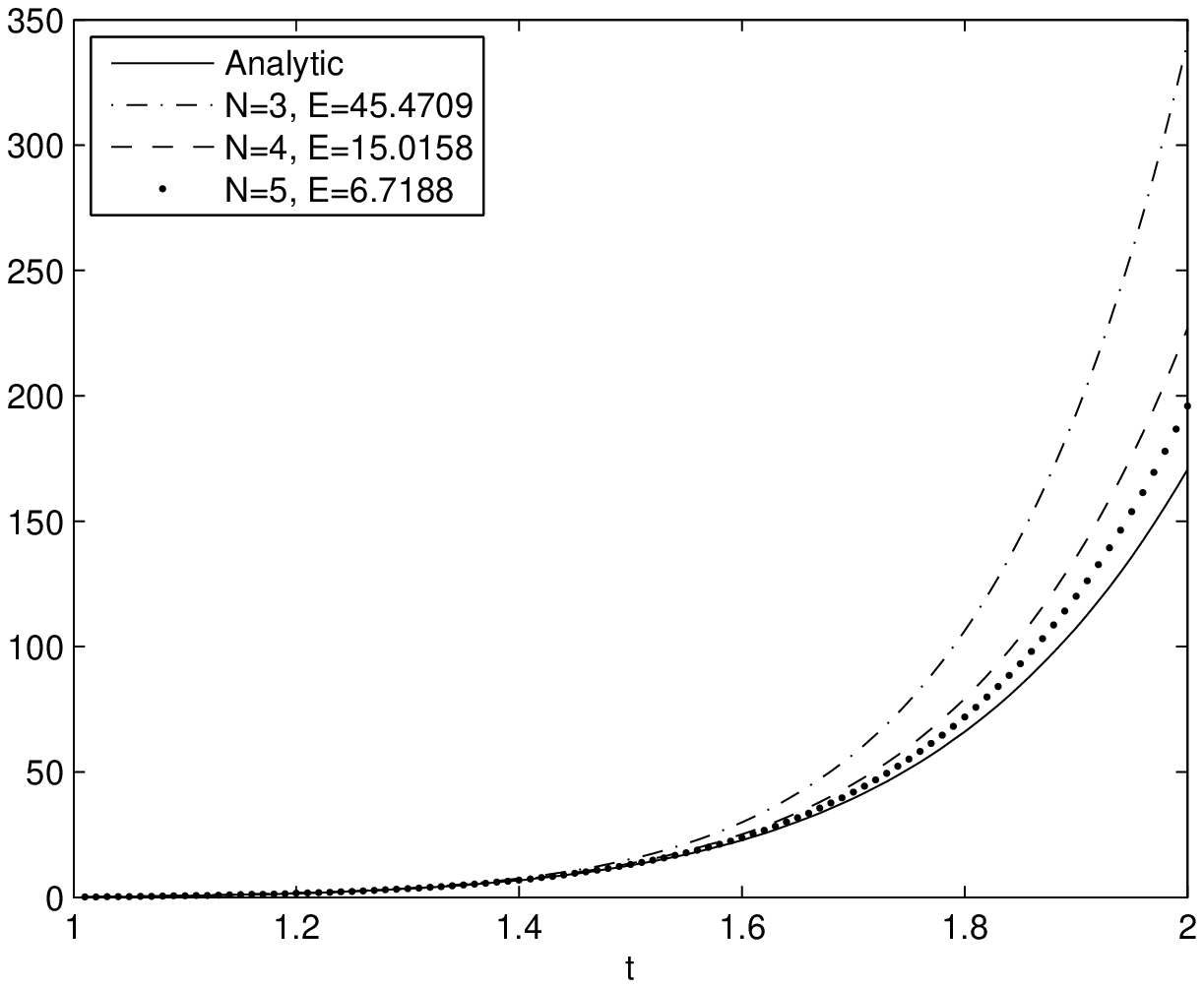}}
\end{center}
\caption{Analytic vs. numerical approximation for $n=2$.}
\end{figure}

Another way to obtain different expansion formulas is to vary $n$.
To exemplify, we choose the previous test functions $x_i$, for $i=1,2,3,4$,
and consider the cases $n=2,3,4$ with $N=5$ fixed. The results are shown
in Figures~\ref{ExpLntNfixed}, \ref{Exp1Nfixed}, \ref{Expt4Nfixed}
and \ref{Expt9Nfixed}. Observe that as $n$ increases, the error may increase.
This can be easily explained by analysis of the error formula, and the values
of the sequence $x_{(k,0)}$ involved. For example, for $x_4$ we have
$x_{(k,0)}(t)=9^kt^9$, for $k=0\ldots,n$. This suggests that, when we increase
the value of $n$ and the function grows fast, in order to obtain a better
accuracy on the method, the value of $N$ should also increase.
\begin{figure}[ht!]
\begin{center}
\subfigure[$\LHIHz(\ln t)$]{\label{ExpLntNfixed}\includegraphics[scale=0.5]{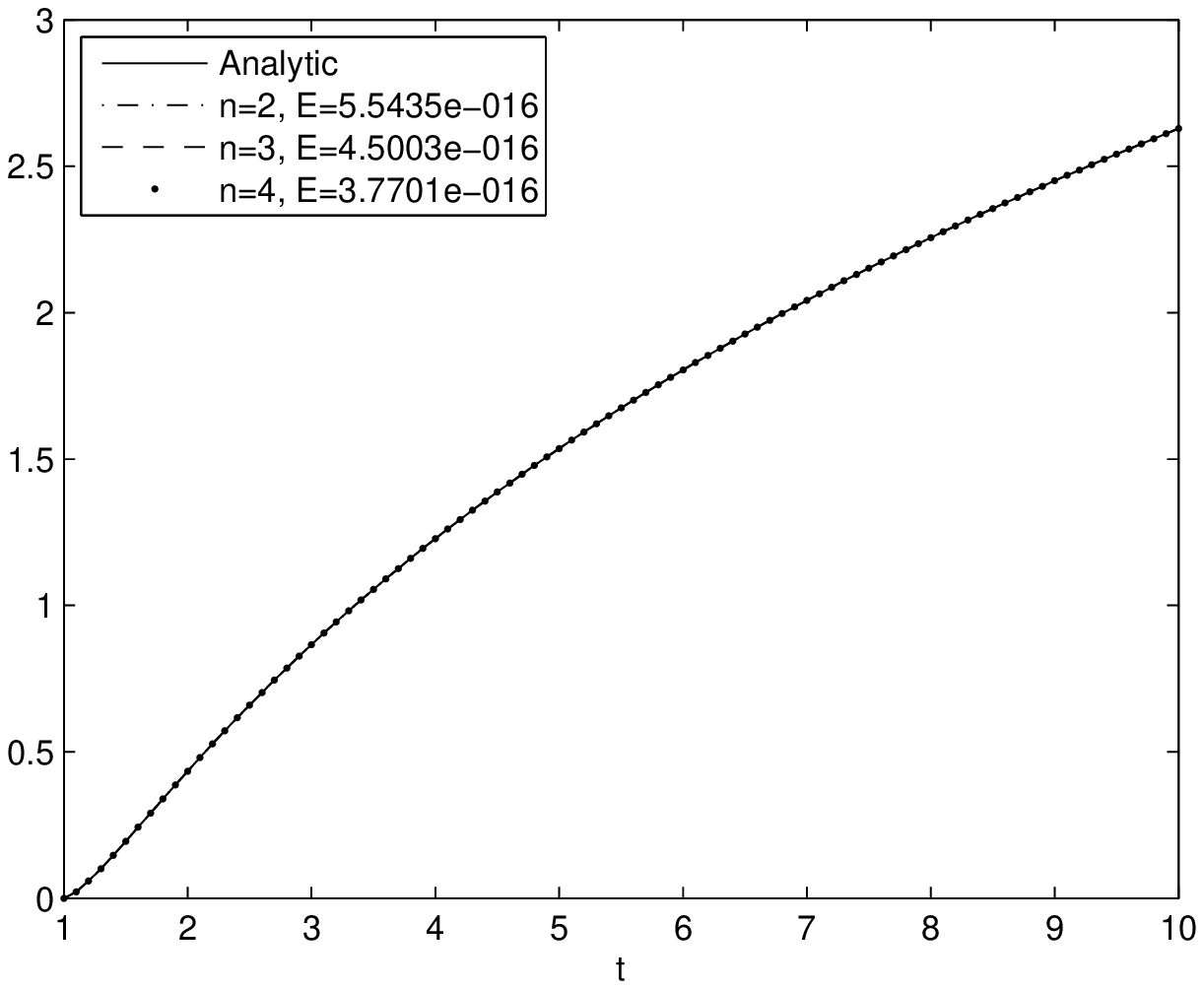}}
\subfigure[$\LHIHz(1)$]{\label{Exp1Nfixed}\includegraphics[scale=0.5]{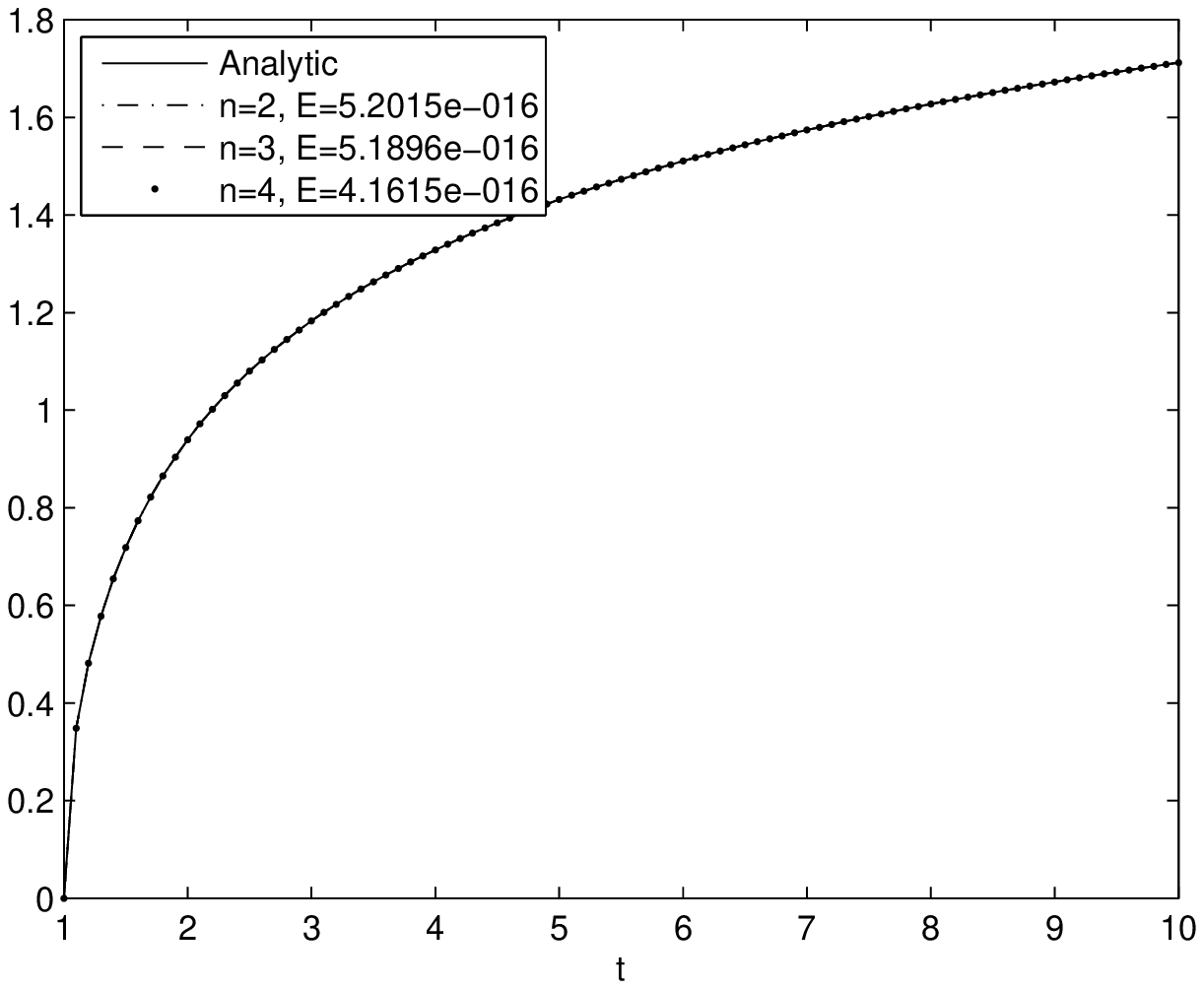}}
\subfigure[$\LHIHz(t^4)$]{\label{Expt4Nfixed}\includegraphics[scale=0.5]{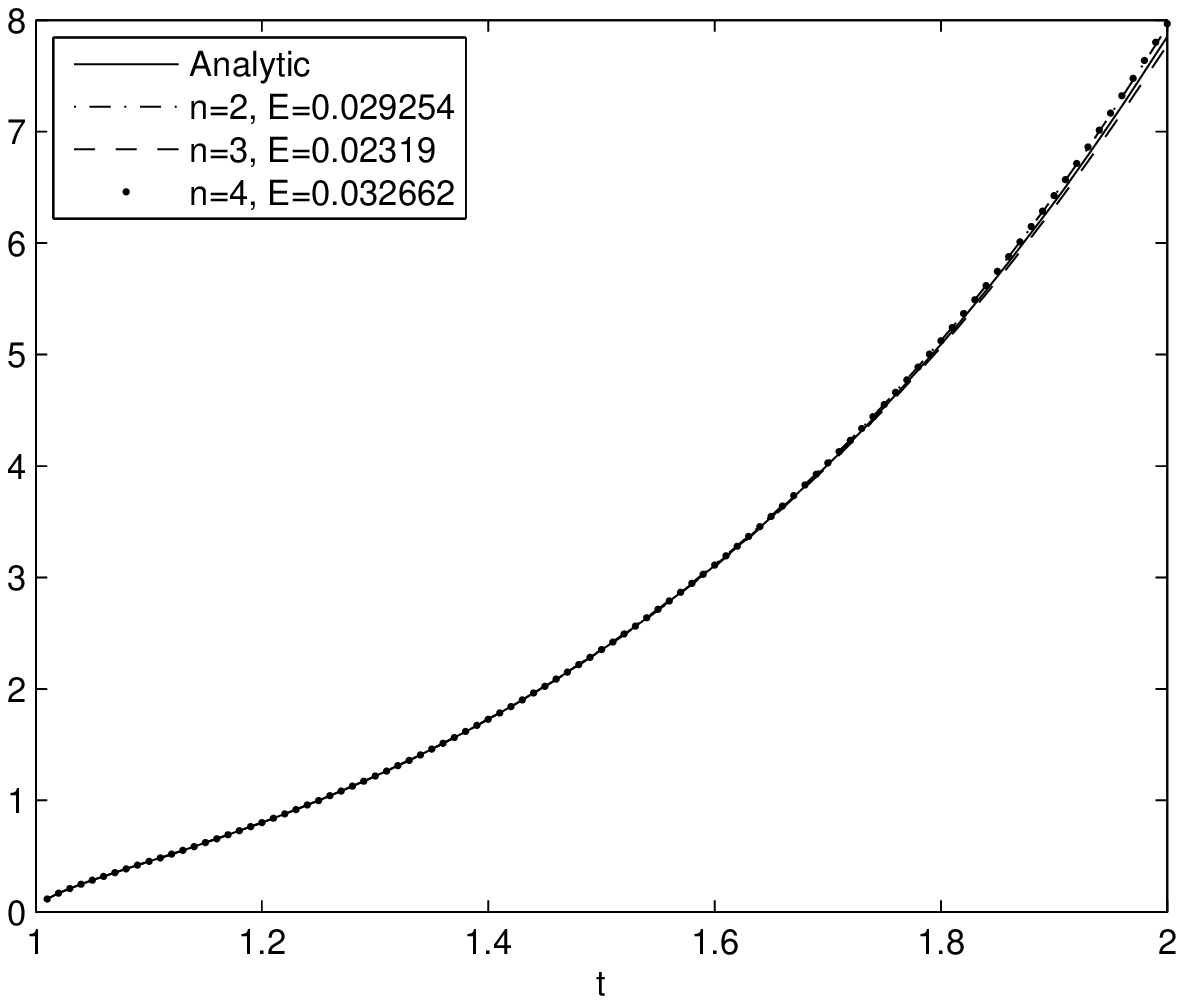}}
\subfigure[$\LHIHz(t^9)$]{\label{Expt9Nfixed}\includegraphics[scale=0.5]{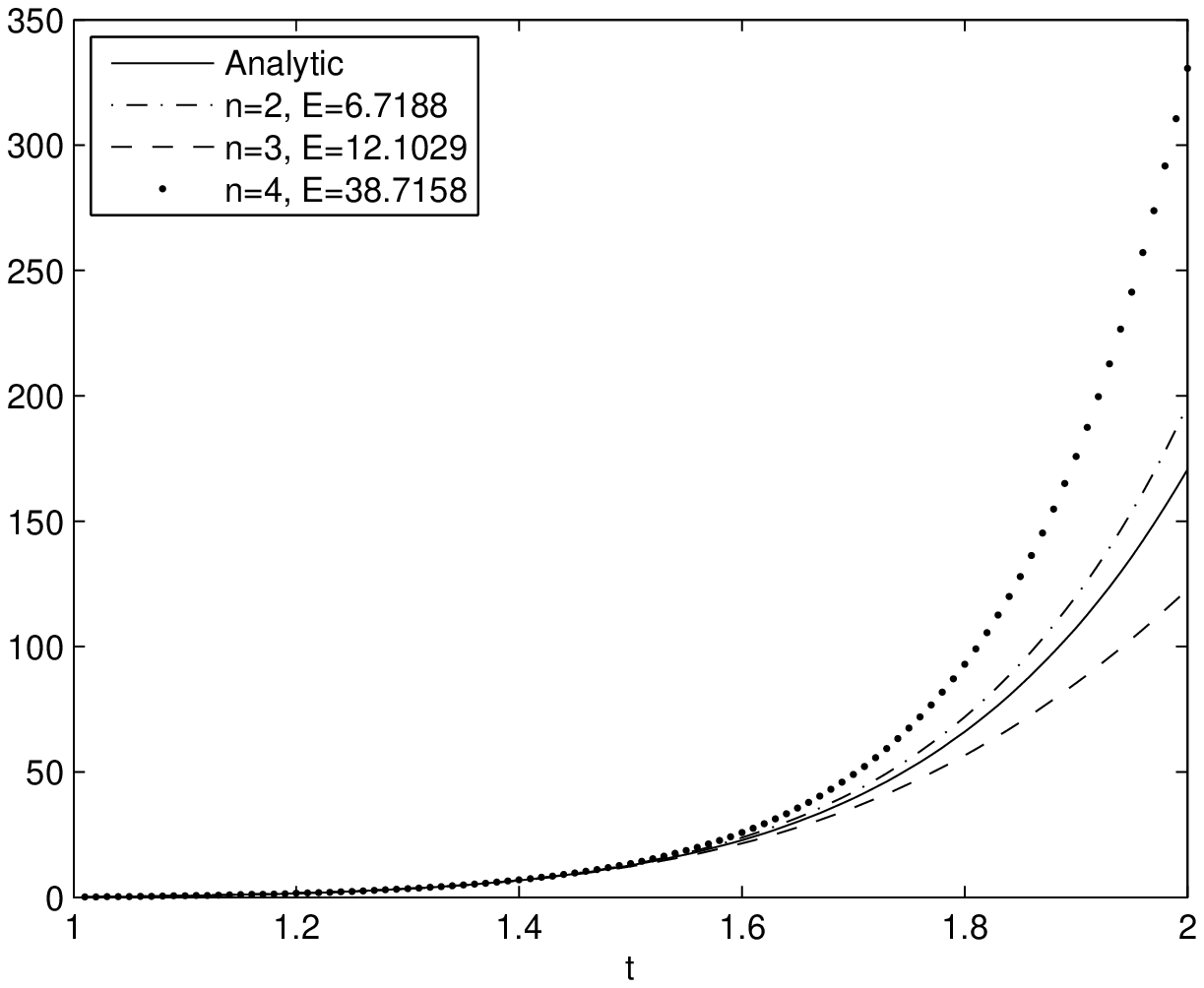}}
\end{center}
\caption{Analytic vs. numerical approximation for $n=2,3,4$ and $N=5$.}
\end{figure}

We now proceed with some examples for the Hadamard fractional derivatives.
The test functions are the same as before, and in Figures~\ref{ExpDerLnt},
\ref{ExpDer1}, \ref{ExpDert4} and \ref{ExpDert9} we exemplify the results.
In this case,
$$
\begin{array}{ll}
\LHDHz x_1(t)&=\displaystyle\frac{\sqrt{\ln t}}{\Gamma(1.5)},\\
\LHDHz x_2(t)&=\displaystyle\frac{1}{\Gamma(0.5)\sqrt{\ln t}},\\
\LHDHz x_3(t)&\displaystyle\approx \frac{1}{\Gamma(0.5)\sqrt{\ln t}}
+\frac{0.8862269255}{\Gamma(0.5)}4t^4 \mbox{erf}(2\sqrt{\ln t}),\\
\LHDHz x_4(t)&\displaystyle\approx \frac{1}{\Gamma(0.5)\sqrt{\ln t}}
+\frac{0.5908179503}{\Gamma(0.5)}9t^9 \mbox{erf}(3\sqrt{\ln t}).
\end{array}
$$

\begin{figure}[ht!]
\begin{center}
\subfigure[$\LHDHz(\ln t)$]{\label{ExpDerLnt}\includegraphics[scale=0.5]{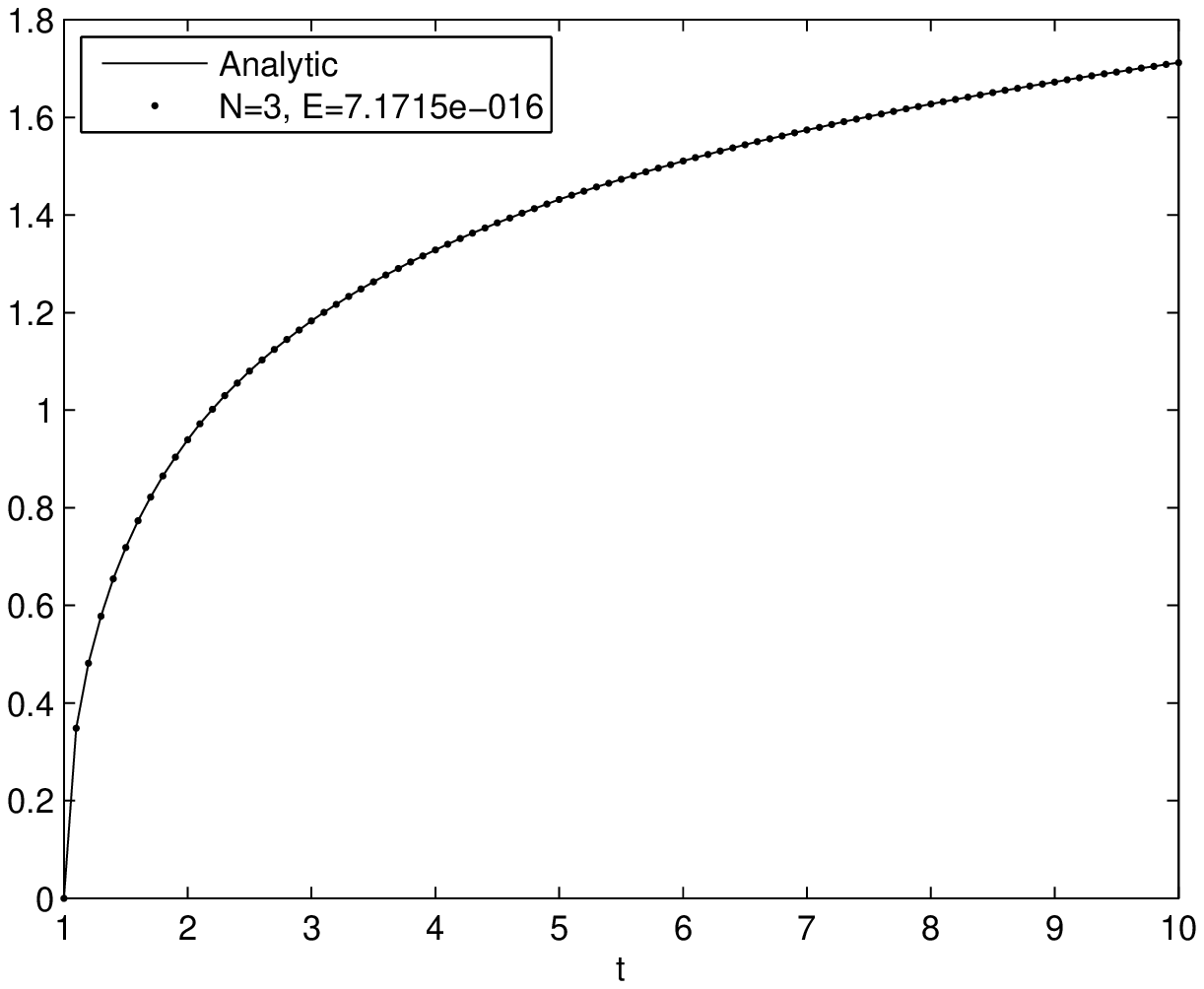}}
\subfigure[$\LHDHz(1)$]{\label{ExpDer1}\includegraphics[scale=0.5]{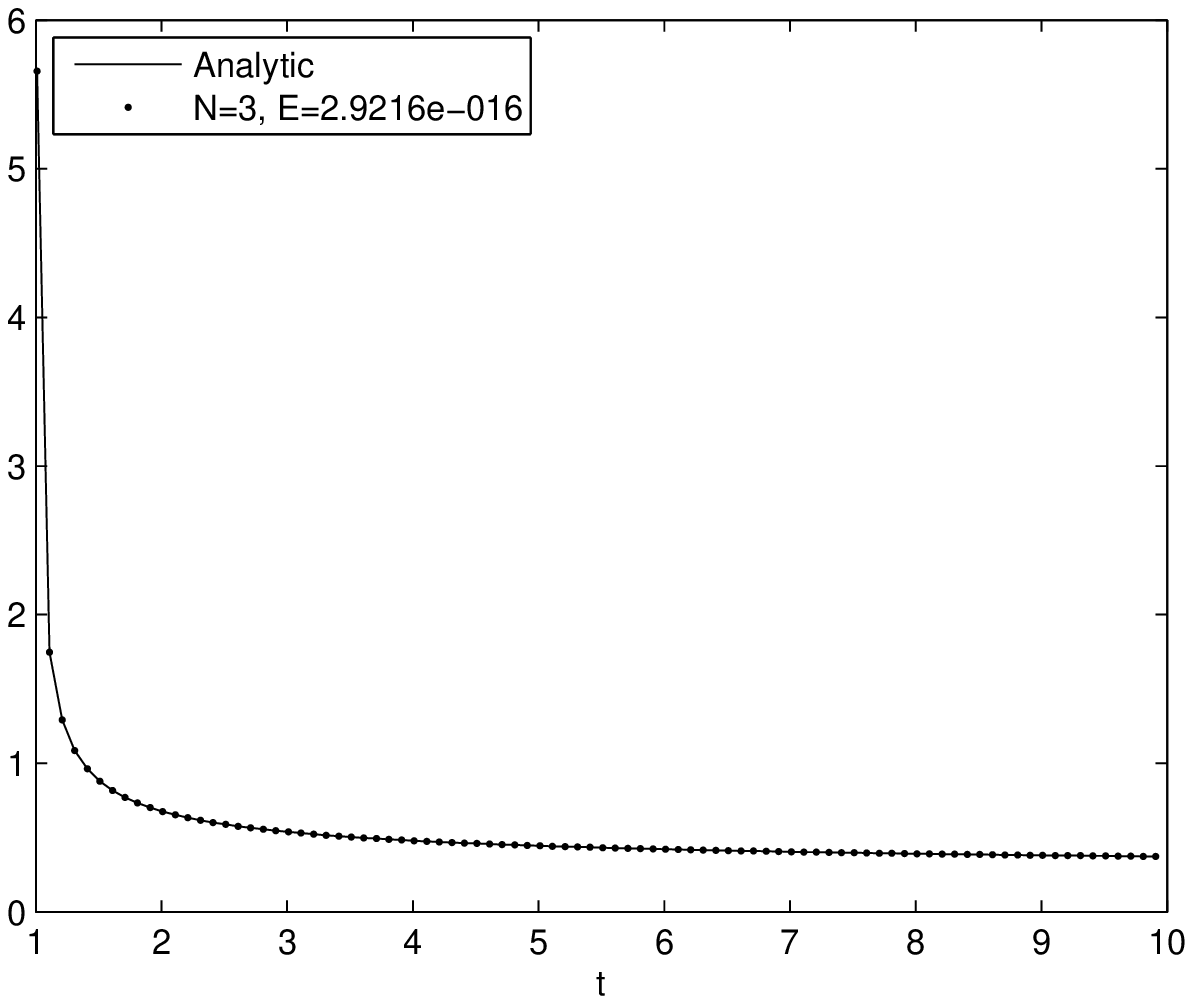}}
\subfigure[$\LHDHz(t^4)$]{\label{ExpDert4}\includegraphics[scale=0.5]{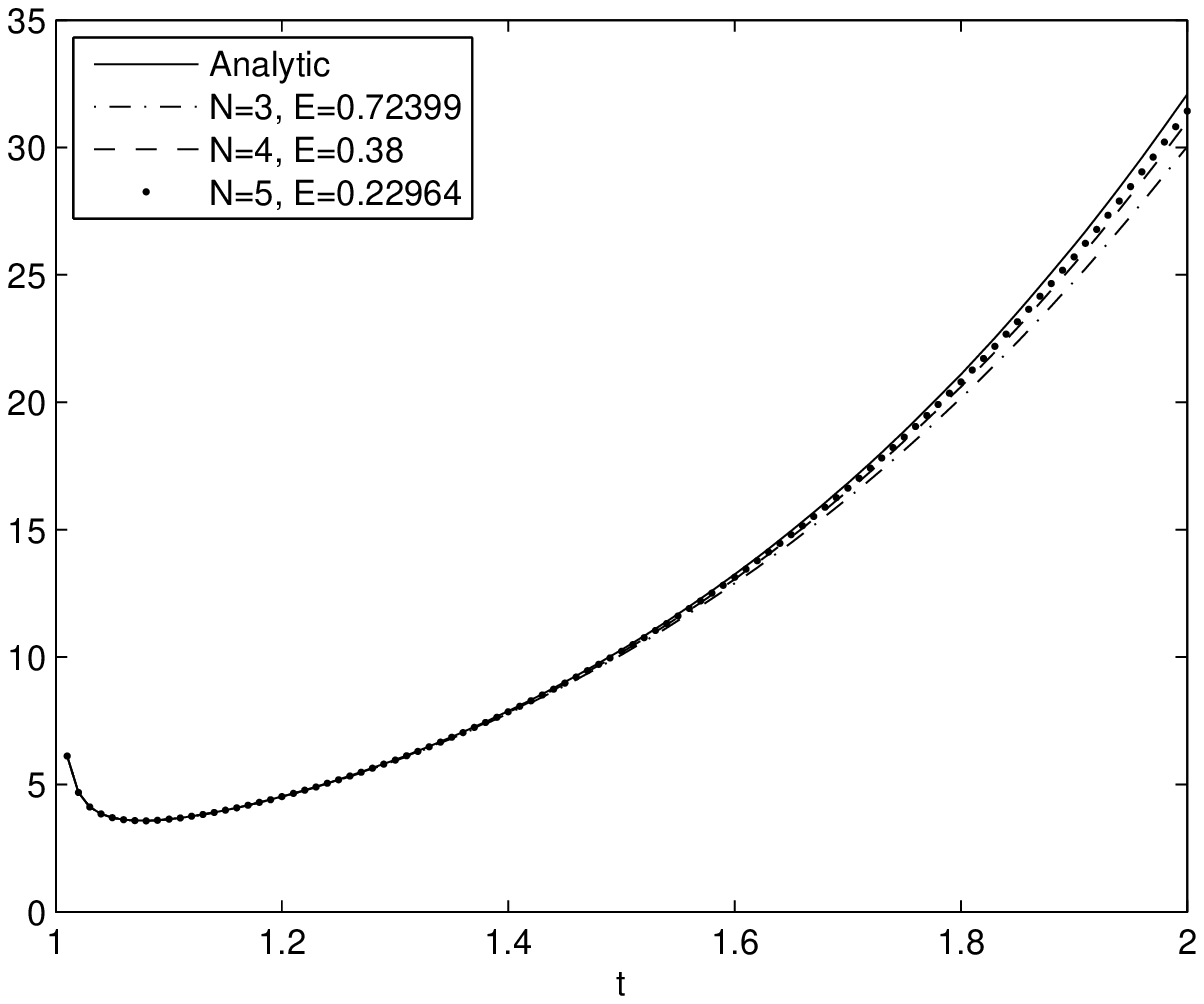}}
\subfigure[$\LHDHz(t^9)$]{\label{ExpDert9}\includegraphics[scale=0.5]{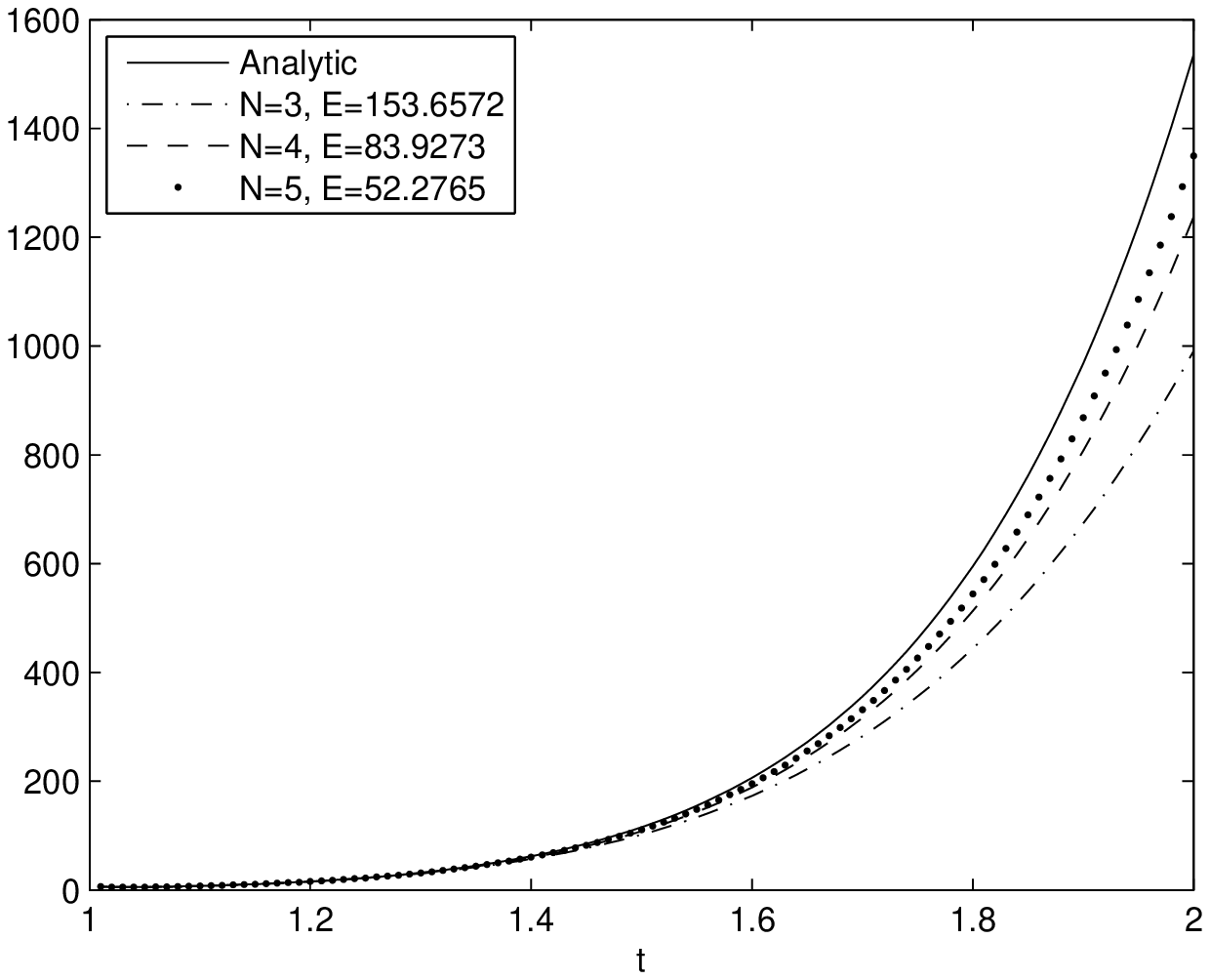}}
\end{center}
\caption{Analytic vs. numerical approximation for $n=2$.}
\end{figure}

One main advantage of this method is that we can replace fractional integrals
and fractional derivatives as a sum of integer/classical derivatives, and by doing
this we are rewriting the original problem, that falls in the theory
of fractional calculus, into a new one where we can apply the already
known techniques (analytical or numerical) and thus solving it. For example,
when in presence of a fractional integral or a fractional derivative,
with a number of initial conditions, replace the fractional operator by the
appropriate approximation, with the value of $n$
given by the number of initial conditions.

For example, consider the problem
\begin{equation}
\label{FDE1}
\left\{
\begin{array}{l}
\displaystyle\LHDHz x(t)+x(t)
=\frac{\sqrt{x(t)}}{\Gamma(1.5)}+\ln t\\
x(1)=0.
\end{array}
\right.
\end{equation}
Obviously, $x(t)=\ln t$ is a solution for \eqref{FDE1}.
Since we have only one initial condition, we replace the operator
$\LHDHz (\cdot)$ by the expansion with $n=1$ and thus obtaining
\begin{equation}
\label{eq:na:exp}
\left\{
\begin{array}{l}
\displaystyle\left[1+A_0(0.5,N)(\ln t)^{-0.5}\right]x(t)
+A_1(0.5,N)(\ln t)^{0.5}t \dot{x}(t)
+\sum_{p=2}^NB(0.5,p)(\ln t)^{0.5-p}V_p(t)
=\frac{\sqrt{x(t)}}{\Gamma(1.5)}+\ln t,\\
\displaystyle\dot{V_p}(t)=(p-1)(\ln t)^{p-2}\frac{x(t)}{t},
\quad p=2,3,\ldots,N,\\
x(1)=0,\\
V_p(1)=0, \quad p=2,3, \ldots, N.
\end{array}\right.
\end{equation}

In Figure~\ref{FDEn=1} we compare the analytical solution
of the FDE \eqref{FDE1} with the numerical result
for $N=2$ in \eqref{eq:na:exp}.
\begin{figure}
\begin{center}
\includegraphics[width=10cm]{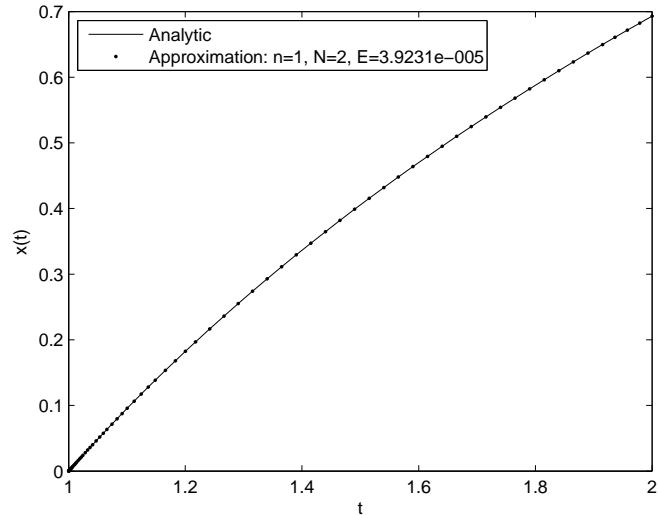}\\
\caption{Analytic vs. numerical approximation
for the FDE \eqref{FDE1} with one initial condition.}\label{FDEn=1}
\end{center}
\end{figure}


\section*{Acknowledgments}

Work supported by {\it FEDER} funds through {\it COMPETE}
--- Operational Programme Factors of Competitiveness
(``Programa Operacional Factores de Competitividade'') ---
and by Portuguese funds through the {\it Center for Research
and Development in Mathematics and Applications} (University of Aveiro)
and the Portuguese Foundation for Science and Technology
(``FCT -- Funda\c{c}\~{a}o para a Ci\^{e}ncia e a Tecnologia''),
within project PEst-C/MAT/UI4106/2011
with COMPETE number FCOMP-01-0124-FEDER-022690.
Shakoor Pooseh was also supported by the
Ph.D. fellowship SFRH/BD/33761/2009.
The authors are grateful to three referees
for their constructive and helpful comments
and valuable suggestions.



\end{document}